\documentclass[a4paper]{article}
\usepackage{geometry}

\usepackage{palatino}
\usepackage{latexsym}

\usepackage[mathscr]{eucal}
\usepackage{amsmath}
\usepackage{amsthm}
\usepackage{amsfonts}
\usepackage{amssymb}
\usepackage{amscd}
\usepackage{color}
\usepackage{graphicx}
\usepackage{graphics}
\usepackage{pifont}
\usepackage{subfigure}
\usepackage{todonotes}

\theoremstyle{plain}
\newtheorem{theorem}{Theorem}
\newtheorem{remark}{Remark}

\newtheorem{proposition}{Proposition}

\newcommand{\ii}{\mathrm{i}}

\newcommand{\Zz}{{\cal  Z} }

\newcommand{\DE}{\Delta_{ex}}
\newcommand{\DI}{\Delta_{in}}

\newcommand{\gM}{{\cal G}}

\title{Embedding the Grushin Cylinder in ${\bf R}^3$\\and Schroedinger evolution }

\begin{document}

\author{Ivan Beschastnyi\footnote{CIDMA - Centro de I\&D em Matematica e Aplicações, Departamento de Matematica Campus Universitario de Santiago, 3810-193, Aveiro, Portugal. ivan.beschastnyi@sorbonne-universite.fr} , Ugo Boscain\footnote{ CNRS, Laboratoire Jacques-Louis Lions, Sorbonne Université, Université de Paris, Inria, Boîte courrier 187, 75252 Paris Cedex 05 Paris, France. ugo.boscain@sorbonne-universite.fr} , Daniele Cannarsa\footnote{Department of Mathematics and Statistics, University of Jyv\"askyl\"a, 40014 Jyv\"askyl\"a, Finland daniele.d.cannarsa@jyu.fi},
Eugenio Pozzoli\footnote{Dipartimento di Matematica, Università di Bari, I-70125 Bari, Italy. e-mail: eugenio.pozzoli@uniba.it}} 

\maketitle

\begin{abstract}

We consider the evolution of a free quantum particle on the Grushin cylinder, under different type of quantizations. In particular we are interested to understand if the particle can cross the singular set i.e., the set where the structure is not Riemannian. We consider intrinsic and extrinsic quantizations where the last ones are obtained by embedding the Grushin structure isometrically in ${\bf R}^3$ (with singularities). As a byproduct we provide  formulas to embed the Grushin cylinder in ${\bf R^3}$ that could be useful for other purposes. Such formulas are not global, but permit to study the embedding arbitrarily close to the singular set.

We extend these results to the case of $\alpha $-Grushin  cylinders.

\end{abstract}
    
\section{Introduction}

The Grushin  plane is the almost-Riemannian structure  on  ${\bf R}^2$ for which a global orthonormal frame in the coordinates $(x,y)$ is given by
$$
X_1(x,y) =\partial_x,~~~~X_2(x,y)=x\,\partial_y.
$$
When the  $y$ direction is compactified on the circle $S^1$  we talk of the {\em Grushin cylinder}. In the following, when we are talking both about the Grushin plane and the Grushin cylinder we will refer to the {\em Grushin manifold} and we will indicate it with $\gM$.

The Grushin manifold has been extensively studied and has very interesting features. Let us list some of them (See for instance \cite{ABS,ABB}).
\begin{itemize}
\item Let $\Zz=\{ (x,y)\in \gM\mid x=0 \}$. On $\gM\setminus\Zz$ the structure  is Riemannian and the Riemannian metric  (in the coordinates $(x,y)$)  is
 \begin{equation}
 {\bf g}= \left(\begin{array}{cc} 1&0\\0&\frac{1}{x^2}  \end{array}\right).
 \label{befana}
  \end{equation}
Actually one could say that the Grushin manifold is the generalized Riemannian structure for which the inverse of the metric is well defined but not invertible on $\Zz$, 
$$
 {\bf g}^{-1}= \left(\begin{array}{cc} 1&0\\0&x^2  \end{array}\right).
 $$
  
The element of area and the Gaussian curvature are given by
$$
dA= \frac{1}{|x|}dx\,dy,~~~~~~K=-\frac2{x^2}
.$$
Notice that they diverge while approaching the singular set.

\item Although all Riemannian quantities diverge when $x\to0$, the geodesics are well defined, can cross the singular set with no singularities (see Figure \ref{f1}) and are easily computed via the Pontryagin maximum Principle. It is interesting to notice that geodesics can have conjugate points even if the curvature is always negative at each point where it is defined.

\end{itemize}

\begin{figure}
\begin{center}
\includegraphics[width=0.7\linewidth]{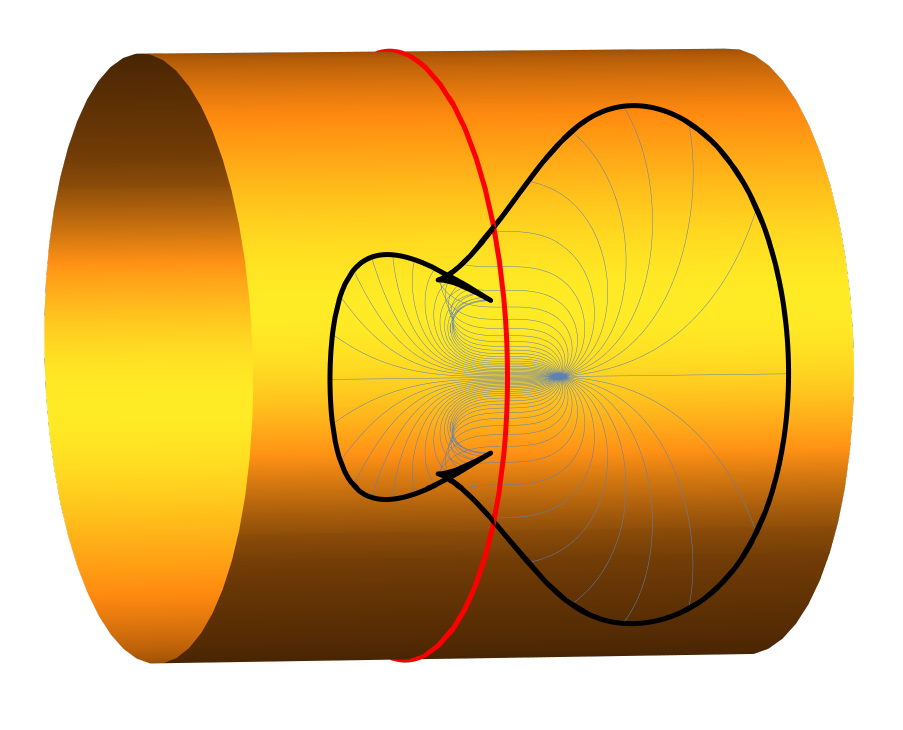}
\caption{Geodesics on the Grushin cylinder starting from the point 
$(1/4,0)$. The end point of geodesics for $T=1.3$ is also shown. Notice that geodesics cross smoothly the singular set (the red circle). Moreover, notice the presence of conjugate points even if the Gaussian curvature is always negative (where it is defined), cf. \cite{ABS,ABB}\label{f1}.}
\end{center}
\end{figure}

The Laplace-Beltrami operator on such structures is given by
\begin{equation}
\Delta={\rm div}_{dA} {\rm grad} =\partial_x^2+x^2\partial_y -\frac{1}{x}\partial_x.
\label{1}
\end{equation}
Here the gradient is computed with respect to the Riemannian structure, which is not singular since the inverse of the metric is well defined. The divergence is computed with respect to the element of area $dA$. The singularity in $dA$ produces the diverging first  order term in \eqref{1}.

One of the most interesting and counterintuitive features of the Grushin plane and of the Grushin cylinder is given by the following proposition.
\begin{proposition}[\cite{Boscain-Laurent-2013}]
\label{camillo}
$\Delta$ with domain $C^\infty_0(\gM\setminus\Zz)$ is essentially self-adjoint in $L^2(\gM,dA)$.
\end{proposition}
An equivalent statement is the following. Define
$$
\gM^{+}=\{ (x,y)\in \gM\mid x>0 \},~~~ \gM^{-}=\{ (x,y)\in \gM\mid x<0 \}.
$$
\begin{proposition}[\cite{Boscain-Laurent-2013}]
\label{camillobis}
$\Delta$ with domain $C^\infty_0(\gM^+)$ is essentially self-adjoint in $L^2(\gM^+,dA)$. The same result holds for $\Delta$ on $\gM^-$. 
\end{proposition}

This type of results has been extensively studied (see for instance \cite{Prandi-Rizzi-Seri-2016,GMP-Grushin-2018,Franceschi-Prandi-Rizzi-2017}) and they hold in a much general context \cite{Boscain-Laurent-2013,ivan-gruppoidi}.

As a consequence of Proposition \ref{camillo}  or \ref{camillobis}, the Cauchy problems for the heat and the Schroedinger equations
$$
\left\{\begin{array}{l} \partial_t \phi(t,p)=\Delta\phi(t,p),\\  \phi(0,\cdot)=\phi_0(\cdot)\in L^2(\gM,dA) \end{array}\right.~~~~~~~~
\left\{\begin{array}{l} i\hbar\partial_t \psi(t,p)=-\frac{\hbar^2}{2m}\Delta\psi(t,p),\\  \psi(0,\cdot)=\psi_0(\cdot)\in L^2(\gM,dA) \end{array}\right.
$$
do not permit any communication between $\gM^+$ and $\gM^-$ since they are are well defined on $\gM^+$ (respectively $\gM^-$).
More precisely if $\phi_0$ and $\psi_0$ are supported in $\gM^+$ (respectively $\gM^-$), then $\phi(t)$ and $\psi(t)$ are supported in $\gM^+$ (respectively $\gM^-$) for any $t\geq0$. 

Such result is rather counterintuitive since geodesics smoothly cross the singular set and the intuition  for the heat equation (respectively for the Schroedinger equation) that quick stochastic particles  (respectively  a well confined wave packet) ``follow'' classical geodesics does not work. For this reason it would be nice to have a more  ``dynamical'' explanation of the absence of communication between $\gM^+$ and $\gM^-$.

For the heat equation an explanation in terms of a limit random walk is possible and is related to the interpretation of \eqref{1} in terms of Bessel processes \cite{boscain-neel}. For the Schroedinger equation such results are more difficult to interpret.

The difficulty of the interpretation of the result for the Schroedinger equation is actually complicated by the fact that ``what is the Schroedinger equation describing the evolution of a free quantum particle in a Riemannian manifold'' has no a unambiguous answer, even without the difficulties introduced by deep singularities present in the Grushin manifold.

As a matter of fact, there are several different approaches that make appearing curvature in the operator describing the quantization of a free particle in a Riemannian manifold. We divide these approaches in two categories: {\em intrinsic} and {\em extrinsic}.\\[2mm]
{\bf Intrinsic approaches}\\
 Most of coordinate invariant quantization procedures on a Riemannian manifold $M$ modify the quantum Hamiltonian  $-\frac{\hbar^2}{2m}\Delta$, where $\Delta$ is the Laplace-Beltrami operator, 
by a correction term depending on the scalar curvature $R$. On a Riemannian manifold  of dimension $2$  the scalar curvature is twice the Gaussian curvature $K$ and the modified Schroedinger equation is of the form
$$
\ii\hbar\,\partial_t\psi(t,p)=-\frac{\hbar^2}{2m}\DI \psi(t,p) ,
~~~p\in M\mbox{ (of dimension 2)}
$$
where
$$
\DI=\Delta-cK(p)
$$
and  $c\geq0$ is a constant.  Values given in the literature include:
\begin{itemize}
\item path integral quantization: $c=1/3$ and  $c=2/3$ in  \cite{driver-22}, $c=1/2$ in \cite{driver-20};
\item covariant Weyl quantization: $c\in[0,2/3]$ including conventional Weyl quantization ($c=0$) in \cite{fulling};
\item geometric quantization for a real polarization: $c=1/3$ in  \cite{driver-97};
\item finite dimensional approximations to Wiener Measures: $c=2/3$ in \cite{driver}.
\end{itemize}

We refer to \cite{driver,fulling} for interesting discussions on the subject.

In \cite{Ivan-Ugo-Eugenio-2020,ivan-gruppoidi} it has been proven that the only value of $c$ for which there is no communication between $\gM^+$ and $\gM^-$ is $c=0$.  

\begin{proposition}
\label{c=0}
Let $c\geq0$.  The operator  $\displaystyle\Delta-c K(x,y)=(\partial_x^2+x^2\partial_y -\frac{1}{x}\partial_x)+c\frac{2}{x^2}$
with domain $C^\infty_0(\gM\setminus\Zz)$ is essentially self-adjoint in $L^2(\gM,dA)$ if and only if $c=0$.
\end{proposition}
For $c\neq0$ there exists a self-adjoint extension of $\Delta-c K$ which permits a communication between $\gM^+$ and $\gM^-$.

As for Proposition \ref{camillo}, Proposition \ref{c=0} has been proven in the more general context of 2-di\-men\-sional, step 2 almost-Riemannian manifolds.
\\[2mm]
{\bf Extrinsic approaches}\\
There are  other approaches to the quantization process on Riemannian manifolds that provide correction terms depending on the curvature. 
Let $S_\varepsilon$ be a $\varepsilon$-tubular neighborhood of an orientable surface $S$ in ${\bf R}^3$:
\begin{equation}
S_\varepsilon = \{p+\tau\,n(p) \,:\, p\in S, \tau\in[-\varepsilon,\varepsilon]\},\label{intorno}
\end{equation}
where $n(p)$ is the normal vector to $S$ at a point $p\in S$. 
Consider the standard Euclidean Laplacian with Dirichlet boundary conditions, then for $\varepsilon\to0$, after a suitable renormalization, one gets an operator containing a correction term depending on the Gaussian curvature and the square of the mean curvature, that is the quantum Hamiltonian,
$$
-\frac{\hbar^2}{2m}\DE,
$$
where 
$$
\DE=\Delta -K(p)+H^2(p).
$$
More precisely the use of this quantum Hamiltonian is justified by the following result.

\begin{theorem}
\label{tubular}
Let $S$ be an orientable surface in $\mathbf{R}^3$, $K$ and $H$ its Gaussian and mean curvature, and $\Delta$ its Laplace-Beltrami operator. Consider the $\varepsilon$-tubular neighbourhood $S_\varepsilon$ of the surface. Let $\Delta^\varepsilon_{DD}$ be the Laplacian of $S_\varepsilon$ with Dirichlet boundary conditions on the two sides of the neighbourhood. Consider $\Delta^\varepsilon_{DD}$ with domain 
\begin{equation}
\label{eq:domain}
D(\Delta^\varepsilon_{DD}) = \left\{\frac{1}{\sqrt \varepsilon} \cos \left(\frac{\pi \tau}{2\varepsilon}\right)\frac{ \psi}{\sqrt{1-2H\tau +K\tau^2}} \, : \, \psi\in C^\infty_c(S) \right\}.
\end{equation}
Then  $\Delta^\varepsilon_{DD}$ has the following asymptotic behavior as $\varepsilon$ tends to zero:
\begin{equation}
\label{eq:formula}
\Delta^\varepsilon_{DD}=-\Big(\dfrac{\pi}{2\varepsilon}\Big)^2+ \underbrace{\Delta-K(p)+H^2(p)}_{=:\DE}+O(\varepsilon).
\end{equation}
\end{theorem}

Theorem \ref{tubular} is an adaptation of a more general result presented in \cite{David3}. See also \cite {David2,Lampart}.
In these works the curvatures are assumed to be bounded. In our case, since we want to apply the result to the Grushin manifold, we localized it on compacts. The term $-K(p)+H^2(p)$ is usually called the {\em effective potential}.

\medskip
Let us give a motivation for the formula~\eqref{eq:formula} and the domain of $D(\Delta^\varepsilon_{DD})$.  The Euclidean volume $\omega$ on $S_\varepsilon$ (cf. \ref{intorno})  has the form
$$
\omega =h(p,\tau)d\tau \wedge dA.
$$
Then it is not difficult to see (see for instance~\cite{tubes}) that $h$ admits the development (which is actually exact)
$$
 h(p,\tau)=1-2H(p) \tau + K(p)\tau^2.
$$
This is positive for small $t$.
Moreover the Euclidean Laplacian on $S_\varepsilon$ can be written as
$$
\Delta^\varepsilon = \partial^2_{\tau} +\frac{\partial_\tau h}{h} \partial_\tau + \Delta_{S(\tau)}, 
$$
where $S(\tau)$, $\tau\in[-\varepsilon,\varepsilon]$ is the surface diffeomorphic to $S$ given by 
\begin{equation}
S(\tau) = \{p+\tau\,n(p) \,:\, p\in S\},\label{intorno2}
\end{equation}
and $\Delta_{S(\tau)}$ is the Laplace-Beltrami operator on $S(\tau)$. Notice that  $\Delta_{S(0)}=\Delta$, the Laplace-Beltrami operator on $S$.
    
 Now, to work in a Hilbert space which does not depend on $\varepsilon$, first we perform the unitary transform:
$$
U:L^2(S_\varepsilon, d\tau \wedge dA)\to L^2(S_\varepsilon, {h}d\tau\wedge dA),
$$
given by
$$
U: f \mapsto h^{{-}1/2}f.
$$
This gives us
$$
U^{-1}\Delta^{\varepsilon} U = \partial_\tau^2 +  {U^{-1}\Delta_{S(\tau)}U}+\frac{{(}\partial_\tau h{)^{2}}-2h\,\partial^2_\tau h}{4h^2}.
$$
The next step is to rescale the $\tau$ variable by taking $\tau=s\varepsilon$, expand the last expression in powers of $\varepsilon$ and recover this way
\begin{equation}
\label{eq:after_unitary}
U^{-1}\Delta^{\varepsilon} U = \varepsilon^{-2}\partial_s^2 + \Delta + H^2-K+ O(\varepsilon). 
\end{equation}
Thus if we ignore the higher order terms, we get an operator, where one can separate the variables. Hence let us apply $U^{-1}\Delta U$ to
$$
\phi(p,s)= \chi(s)\psi(p), \qquad (p,s)\in  S\times [-1,1] \simeq S_\varepsilon.
$$
As we can see in~\eqref{eq:after_unitary} the first term always diverges. Because of this, it makes sense to choose a specific function $\chi$ in order to be able to compensate for this divergent term. In particular, we can choose $\chi$ to be the first eigenfunction of $\partial^2_s$ on $[-1,1]$ with  Dirichlet boundary conditions, which is given by
$$
\chi(s) = \cos\left(\frac{\pi s}{2} \right).
$$

Reverting the scaling and unitary transformation we arrive at the domain~\eqref{eq:domain} and asymptotics~\eqref{eq:formula} meaning that
if
$$
\varphi(\tau,p) = \cos\left(\frac{\pi \tau}{2\varepsilon} \right)\psi(p){h^{-1/2}(\tau,p)},
$$
we get 
$$
\frac{\langle \varphi,\Delta^{\varepsilon}_{DD}\varphi \rangle}{\langle \varphi, \varphi\rangle}=-\Big(\dfrac{\pi}{2\varepsilon}\Big)^2+ \frac{\langle \psi,\DE\psi \rangle}{\langle \psi, \psi\rangle}+O(\varepsilon).
$$
Here $\Delta^\varepsilon$ is called $\Delta^\varepsilon_{DD}$ since Dirichlet boundary conditions are already imposed. In this formula the scalar product on the left hand side is computed in $L^2(S_\varepsilon,hd\tau\wedge dA)$ and the scalar product on the right hand side is 
computed in $L^2(S,dA)$.

 In the article~\cite{David3} the authors prove that $\Delta^\varepsilon_{DD}+(\pi/2\varepsilon)^2$ converges to $\Delta_{ex}$ in an appropriate norm-resolvent sense and derive consequences for the localization of spectrum as $\varepsilon \to 0$. 

\medskip
Purpose of this paper is to apply the extrinsic approach for the Grushin cylinder trying to isometrically embed it as a surface of revolution as close as possible to $\Zz$, to compute its mean curvature $H$, and hence $\Delta_{ex}$.
As expected from the divergence of the Gaussian curvature $K$ while approaching $\Zz$ we are going to discover that the Grushin  cylinder cannot  be embedded globally up to the singularity, but only in the interval $x\in[1,+\infty[$. Actually the singularity of the embedding reflects in an explosion of the mean curvature $H$ for $x\to1$. We call this embedding \emph{the Grushin trumpet bell} (see Figure \ref{f2}). We then study the self-adjointness of $\DE$ on $x\in]1,+\infty[$ and we prove that  $\DE$ is not essentially self-adjoint in $L^2$, with respect to the volume $dA$.

In order to go closer to the singular set $\Zz$ we also study other type of embeddings in which the surface is wrapped $n^2$ times around its axis of revolution  and for which the embedding is possible for $x\in[\frac1n,+\infty[$ (\emph{Grushin $n^2$-winded bell}, see Figure \ref{f3}).

Finally we generalize this study to the $\alpha$-Grushin cylinder 
which is the generalized Riemannian structure on   ${\bf R}\times S^1$ for which a global orthonormal frame  is given by
$$
X_1(x,y) =\partial_x,~~~~X_2(x,y)=|x|^\alpha\,\partial_y,~~~\alpha\in{\bf R}.
$$
The $\alpha$-Grushin cylinder has been studied with different objectives in several papers (see for instance \cite{Boscain-Prandi-JDE-2016,beauchard1,Gallone-Michelangeli-Pozzoli-2020,cirillo} and references therein).

\subsection{Historical remarks}

The Grushin metric has been introduced by Baouendi \cite{baouendi} and Grushin \cite{grusin1} at the end of sixties. At the time people were interested in degenerate elliptic operators, and in particular in the following operator in ${\bf R}^2$
\begin{equation}
\partial_x^2+x^2\partial_y^2.\label{gabbia}
\end{equation}
Such operator contains the metric \eqref{befana} in the sense that it can be written as $\sum_{j,k=1}^2 g^{jk}\partial_j\partial_k$ where $(g^{jk})$ is the inverse of the metric \eqref{befana}.
 It is interesting to notice that the operator \eqref{gabbia} it is not the Laplace-Beltrami operator (i.e., the divergence w.r.t. the Riemannian volume of the gradient) of the metric \eqref{befana}. Notice that, accidentally, the operator \eqref{gabbia} is symmetric w.r.t. the Lebesgue measure of the plane. However  in general if $\{X_1,X_2\}$ is a (generalized) orthonormal frame on a 2D manifold, the operator sum of squares  $(X_1)^2+ (X_2)^2$ is not  symmetric w.r.t. any volume.

\section{Embedding the Grushin half-cylinder in ${\bf R}^3$ as a surface of revolution}

Consider the manifold ${\bf R}^+\times S^1$ with the Riemannian metric  
$$
 {\bf g}= \left(\begin{array}{cc} 1&0\\0&g(x)^2 \end{array}\right),~~~(x,y)\in  {\bf R}^+\times S^1.
 $$
 Under the condition $g'(x)\leq1$, this Riemannian manifold can be realized isometrically 
as the surface of revolution 
\begin{equation}
\left\{
\begin{array}{l}
z_1=g(x)\cos y\\
z_2=g(x) \sin y\\
z_3=h(x), 
\end{array}
\right.\label{revolution}
\end{equation}
where $h(x)=z_{30}+\int_{x_0}^x\sqrt{1-g'(s)^2}\,ds$. At the points $x$ such that $g(x)=0$ the system of coordinates $(x,y)$ is singular (in a similar way in which polar coordinates on the plane are singular in $(0,0$)).

Here $x_0$ and $z_{30}$ are two  constants (of which only one is independent) that fix how the surface is placed on the $z$ axis ($z_{30}=h(x_0)$). 
For this embedding we have the formula for the Gaussian curvature
$$
K=-\frac{h'(x)}{g(x)}(h'(x)g''(x)-h''(x)g'(x)),
$$
and the formula for the mean curvature
$$
H=\frac{1}{2g(x)}\Big(h'(x)-g(x)(h'(x)g''(x)-h''(x)g'(x))\Big).
$$

Notice that the sign of $H$ depends on the choice of the unit normal. However, in the following, nothing depends on this choice  since only $H^2$ appears.
For the Grushin half-cylinder we have $g(x)=1/x$. Hence  $g'\leq1$ for  $x\geq1$ and the embedding takes the form (here $x_0$ and $z_{30}$ are fixed  in such a way that $z_3(1)=1$),
\begin{equation}
\left\{
\begin{array}{l}
z_1=\frac1x \cos y\\
z_2=\frac1x \sin y\\
z_3=1+ \int_{1}^x\sqrt{1-\frac{1}{s^4}}\,ds.
\end{array}
\right.\label{revolution2}
\end{equation}
The expression for $z_3$ can be computed explicitly in terms of special functions, but this is not relevant here. This embedding is actually global for $x\geq1$ (see Figure \ref{f2}). In the following we call this embedding the {\em Grushin trumpet bell}.

\begin{figure}
\begin{center}
\includegraphics[width=1.1\linewidth]{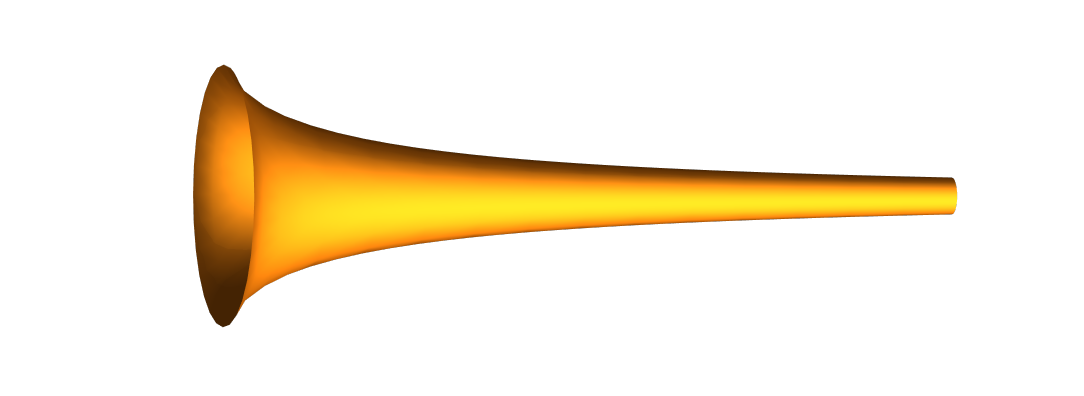}
\caption{The Grushin trumpet bell, i.e., an embedding of the Grushin cylinder on ${\bf R^3}$. In 
 the figure the axis of revolution ($z_3$) is horizontal to render the figure more comparable with Figure \ref{f1}. The bell reaches its largest sectional diameter at the singularity of the embedding ($x=1$). \label{f2}}
\end{center}
\end{figure}

For the Grushin trumpet bell we have that 
$$
H=\frac{x^4-3}{2 x\sqrt{x^4-1}},
$$
and the Gaussian curvature (which does not depend on the embedding) is 
$$
K=-\frac{2}{x^2}.
$$
It follows that the extrinsic Laplacian is  
\begin{align}
\DE&=\Delta -K+H^2= \partial_x^2+x^2\partial_y -\frac{1}{x}\partial_x -\big(-\frac{2}{x^2} \big)+\Big(\frac{x^4-3}{2 x\sqrt{x^4-1}} \Big)^2\nonumber\\
&=
\partial_x^2+x^2\partial_y -\frac{1}{x}\partial_x +\frac{\left(x^4+1\right)^2}{4 x^2 \left(x^4-1\right)}.
\end{align}

\subsection{Study of the self-adjointness of the extrinsic Laplacian}
The next Proposition states that $\DE$ is not essentially self-adjoint $(1,\infty)$ namely 
up to the singularity of the embedding. Actually this is due to the fact that the effective potential
$$
-K+H^2=+\frac{\left(x^4+1\right)^2}{4 x^2 \left(x^4-1\right)}=\frac{1}{4(x-1)}+O(1),~~~\mbox{ for }x\to1
$$
and the operator $-\frac{d^2}{dx^2}-\frac{1}{4(x-1)}$ is in the limit circle case at $x=1$.

\begin{proposition}\label{prop:not-self-adj}
$\DE$ with domain $C^\infty_0((1,\infty)\times S^1)$ is not essentially self-adjoint in $L^2((1,\infty)\times S^1,dA)$.
\end{proposition}
\begin{proof}
By making the operator $\DE$ acting on function of the form $\sqrt{x}\,\phi(x,y)$ one immediately obtain that the operator $\DE$ on $L^2((1,\infty)\times S^1,dA)$
is unitarily equivalent to the operator
$$
L_{ex}=\partial_x^2+x^2\partial_y -\frac34\frac{1}{x^2}+\frac{\left(x^4+1\right)^2}{4 x^2 \left(x^4-1\right)},
$$
on $L^2((1,\infty)\times S^1,dx\,dy)$.
By making Fourier transform in the $y$ variable, the self-adjointness of $L_{ex}$ is equivalent to the self-adjointness 
of the operators
$$
\hat{L}_{ex}^k=\partial_x^2-k^2x^2 -\frac34\frac{1}{x^2}+\frac{\left(x^4+1\right)^2}{4 x^2 \left(x^4-1\right)},
$$
for every $k\in{\bf N}$. Now
$$
-k^2x^2 -\frac34\frac{1}{x^2}+\frac{\left(x^4+1\right)^2}{4 x^2 \left(x^4-1\right)}=\frac{1}{4(x-1)}+O(1),~~~\mbox{ for }x\to1.
$$
As a consequence, $-\hat{L}_{ex}^k$ can be seen as a one-dimensional Schroedinger operator $-\frac{d^2}{dx^2}+V(x)$ with potential function $V$ decreasing as $x\to 1^+$, which implies that it is in the limit circle case at $x=1$ (see e.g. \cite[Theorem X.7 \& Problem 7]{rs2}).  Hence  for every $k\in{\bf Z}$, we have that $\hat{L}_{ex}^k$ is not essentially self-adjoint. The result follows.

\end{proof}

\begin{remark}
From the proof of Proposition \ref{prop:not-self-adj}, we can moreover deduce the lack of self-adjointness of every fiber operator $\hat{L}_{ex}^k, k\in{\bf Z}$. In other words, $\Delta_{ex}$ has infinite deficiency indices.   
\end{remark}

\subsection{Other possible embedding}

The embedding presented above is global for $x\in[1,\infty)$. The singularity at $x=1$ is due to the fact that while approaching the singular set at $x=0$ the area is growing too much to be realized as a surface of revolution. Actually the tangent space to the Grushin trumpet bell for $x\to1^+$ becomes parallel to the $(z_1,z_2)$ plane. This is encoded in the divergence of the  mean curvature $H$.

Other local embeddings are actually possible. To be able to get closer to the singular set $\{x=0\}$ one needs to make the area growing more than the previous embedding, for instance, by embedding only a part of the Grushin cylinder or by winding it several times around the $z_3$ axis.

Let us exploit this  idea. Let $n\in {{\bf N}^\ast}$. For $x\in[1/n,\infty)$ and $y\in 
(0,2\pi/n^2)$, the Grushin cylinder is locally isometric to the surface (here $x_0$ and $z_{30}$ are fixed  in such a way that $z_3(1/n)=1/n$),

\begin{equation}
\left\{
\begin{array}{l}
z_1=\displaystyle\frac{1}{n^2x} \cos (n^2 y)\\
z_2=\displaystyle\frac{1}{n^2x}  \sin (n^2 y)\\
z_3=\displaystyle\frac1n+\frac1n \int_{1}^{n x}\sqrt{1-\frac{1}{s^4}}\,ds.
\end{array}
\right.\label{revolutionWIND}
\end{equation}
In the following we call this embedding the {\em Grushin $n^2$-winded  bell}, because if $y\in S^1$, instead of $y\in 
(0,2\pi/n^2)$,  then the surface is winded $n^2$ times around the axis of revolution.

For the Grushin $n^2$-winded bell  the mean curvature is
$$
H_n =\frac{n^4x^4-3}{2 x\sqrt{n^4 x^4-1}}.
$$
\begin{figure}
\begin{center}
\includegraphics[width=1.1\linewidth]{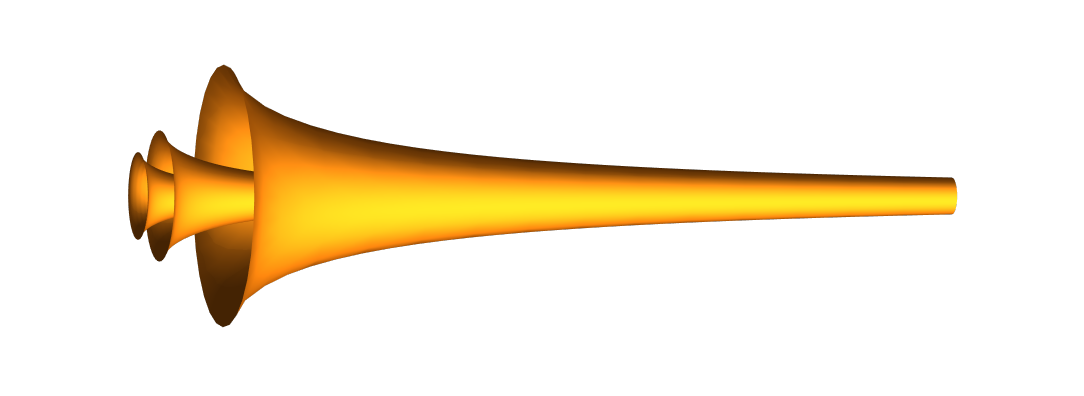}
\caption{Grushin $n^2$-winded bells. These are embeddings of the Grushin cylinder for $(x,y)\in[1/n,\infty)\times (0,2\pi/n^2)$.
As in Figure \ref{f2} the axis of revolution $z_3$ is horizontal. The bells reach their largest diameter at the singularity of the embedding ($x=1/n$). If we take $y\in S^1$ then each bell is winded $n^2$ times.
\label{f3}}
\end{center}
\end{figure}
It follows that the extrinsic Laplacian is  

\begin{align}
\DE^n&=\Delta -K+H^2_n= \partial_x^2+x^2\partial_y -\frac{1}{x}\partial_x -\big(-\frac{2}{x^2} \big)+\Big(\frac{n^4x^4-3}{2 x\sqrt{n^4 x^4-1}}  \Big)^2\nonumber\\
&=
\partial_x^2+x^2\partial_y -\frac{1}{x}\partial_x +
\frac{\left(n^4x^4+1\right)^2}{4 x^2 \left(n^4x^4-1\right)}.
\end{align}
Since for the effective potential is given by 
$$
-K+H^2_n=\frac{\left(n^4x^4+1\right)^2}{4 x^2 \left(n^4x^4-1\right)}=\frac{n}{4(x-\frac1n)}+O(1),~~~\mbox{ for }x\to\frac1n,
$$
as in the case $n=1$ we have the analogous result.

\begin{proposition}
$\DE^n$ with domain $C^\infty_0((\frac1n,\infty)\times S^1)$ is not essentially self-adjoint in $L^2((\frac1n,\infty)\times S^1,dA)$.
\end{proposition}

Unfortunately it seems hard to say anything about the self-adjointness of  $\DE^n$ in the limit $n\to\infty$.

\section{The $\alpha$-Grushin cylinder}

A natural generalization of the Grushin cylinder 
is the $\alpha$-Grushin cylinder 
${\bf R}\times S^1$ endowed with the generalized Riemannian metric 
\begin{equation}
 {\bf g}= \left(\begin{array}{cc} 1&0\\0&\frac{1}{|x|^{2\alpha}}  \end{array}\right),~~~\alpha\in{\bf R}.\label{alpa}
\end{equation}
Notice that the corresponding Riemannian area is 
$$
dA^\alpha=\frac{1}{|x|^\alpha}dxdy.
$$

This class includes as a special case the Grushin cylinder ($\alpha=1$), the standard Euclidean cylinder ($\alpha=0$), a punctured plane ($\alpha=-1$) and a quadratic cusp ($\alpha = -2$).  For $\alpha\in{\bf N}$ the $\alpha$-Grushin cylinder is an almost-Riemannian manifold in the sense of \cite{ABS,ABB}.

Similarly to the previous section, we can find embedding  in ${\bf R}^3$ as surfaces of revolution. For $\alpha\neq-1$, let us define 
$$s_0(\alpha)=|\alpha|^{\frac{1}{1+\alpha}},$$
As a function of $\alpha$ the graph of $s_0$ is pictured in Figure \ref{s_0}.

\begin{figure}[ht]
\begin{center}
\includegraphics[width=0.4\linewidth]{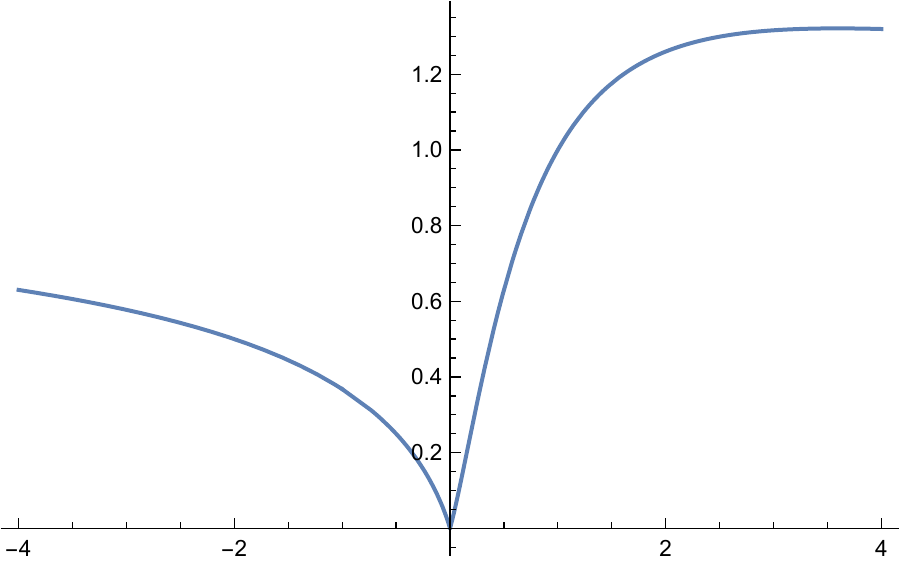}
\caption{The function $s_0(\alpha)$ representing the value of $x$ at which the embedding fails. \label{s_0}}
\end{center}
\end{figure}

For $\alpha>-1$ and $\alpha \neq 0$ the embedding is possible for $x\geq s_0(\alpha)$ (which is the interval for which the square root in the next formula is real) and  we have,
\begin{equation}
\left\{
\begin{array}{l}
z_1=\displaystyle x^{-\alpha} \cos (y)\\
z_2=\displaystyle x^{-\alpha}  \sin (y)\\
z_3=s_0(\alpha)+\displaystyle \int_{s_0(\alpha)}^{x}\sqrt{1-\frac{\alpha^2}{s^{2(1+\alpha)}}}\,ds.
\end{array}
\right.\label{revolution_positive_alpha}
\end{equation}
Here the constants are chosen in such a way that $z_3(s_0(\alpha))=s_0(\alpha)$.

For $\alpha<-1$  the embedding is possible for $x\in(0,s_0(\alpha)]$  and  it is given by

$$
\left\{
\begin{array}{l}
z_1=\displaystyle x^{-\alpha} \cos (y)\\
z_2=\displaystyle x^{-\alpha}  \sin (y)\\
z_3=\displaystyle \int_{0}^{x}\sqrt{1-\frac{\alpha^2}{s^{2(1+\alpha)}}}\,ds.\end{array}
\right.\label{revolution_negative_alpha}
$$
Here the constants are chosen in such a way that $z_3(0)=0$.

For  $\alpha = 0$ the obtain the standard embedding of the cylinder which is possible for every value of $x\in{\bf R}$, 
$$
\left\{
\begin{array}{l}
z_1=\displaystyle \cos (y)\\
z_2=\displaystyle \sin (y)\\
z_3=\displaystyle x.
\end{array}
\right.
$$

For  $\alpha = -1$ we obtain the embedding of a punctured plane in polar coordinates (here we have  $x>0$),
$$
\left\{
\begin{array}{l}
z_1=\displaystyle x \cos (y)\\
z_2=\displaystyle x  \sin (y)\\
z_3=\displaystyle 0.
\end{array}
\right.
$$
Some of these embeddings are pictured in Figures \ref{falfa-1}, \ref{falfa-2}, \ref{falfa-3}, \ref{falfa-4}.

 For every value of $\alpha$, the Gaussian curvature is given by
$$
K= -\frac{\alpha(1+\alpha)}{x^2}.
$$
 For $\alpha \neq -1$, the mean curvature is given by
$$
H=\frac{x^{2(1+\alpha)}-\alpha(1+2\alpha)}{2x \sqrt{x^{2(1+\alpha)}-\alpha^2}},
$$
and for $\alpha = -1$ it is given by the limit of this expression, which is zero. 

For the effective potential we find
\begin{equation}
-K + H^2= \frac{(x^{2(1+\alpha)}+\alpha)^2}{4x^2(x^{2(1+\alpha)}-\alpha^2)} = \frac{1+\alpha}{8 |\alpha|^{\frac{1}{1+\alpha}}(x-|\alpha|^{\frac{1}{1+\alpha}})} + O(1),~~~\mbox{ for }x\to|\alpha|^{\frac{1}{1+\alpha}}.\label{plutonio}
\end{equation}
Hence, once again we find the same lack of self-adjointness.
\begin{proposition}
Consider the generalized Riemannian metric \eqref{alpa} on ${\bf R}\times S^1$. Let $\Delta$ be its Laplace-Beltrami operator  and $\DE^\alpha=\Delta-K + H^2$ its extrinsic Laplacian corresponding to the embeddings described above. We have the following:
\begin{itemize}
\item for $\alpha\geq-1$, $\DE^\alpha$   with domain  $C^\infty_0((s_0(\alpha),\infty)\times S^1)$ is not essentially self-adjoint  in $L^2((s_0(\alpha),\infty)\times S^1,dA^\alpha)$;
\item  for $\alpha<-1$, $\DE^\alpha$   with domain  $C^\infty_0((0,s_0(\alpha))\times S^1)$  is not essentially self-adjoint  in $L^2((0,s_0(\alpha))\times S^1,dA^\alpha)$.
\end{itemize}

\end{proposition}

\begin{remark}
Due to the asymptotic \eqref{plutonio}, boundary conditions are always necessary at $x=s_0(\alpha)$ to give a meaning to the Cauchy problem. 
In the case $\alpha<-1$ the embeddings are defined in the interval $(0,s_0(\alpha)]$. Actually the non-selfadjointness comes from both extremities of this interval (meaning that boundary conditions are necessary at $x=0$ as well).

 To see this 
let us  make a unitary transform $Uf(x,y) = x^{-\alpha/2}f(x,y)$. This gives the operator
$$
L_{ex} = U\Delta_{ex} U^{-1} = \partial_x^2 + x^{2\alpha}\partial_y^2 -\frac{\alpha}{2}\left(1+\frac{\alpha}{2} \right) \frac{1}{x^2} + \frac{(x^{2(1+\alpha)}+\alpha)^2}{4x^2(x^{2(1+\alpha)}-\alpha^2)}.
$$
Taking the Fourier transform in the $y$-variable gives us a family of operators
$$
\hat{L}^k_{ex} = \partial_x^2 - k^2 x^{2\alpha} -\frac{\alpha}{2}\left(1+\frac{\alpha}{2} \right) \frac{1}{x^2} + \frac{(x^{2(1+\alpha)}+\alpha)^2}{4x^2(x^{2(1+\alpha)}-\alpha^2)}, \quad k\in{\bf Z}.
$$
For $\alpha< -1$ we get
$$
\hat{L}^k_{ex} = \partial_x^2 + \left(\frac{1}{4}-k^2\right) x^{2\alpha} + o(x^{2\alpha}),~~~\mbox{ for }x\to 0^+
$$
hence $\hat{L}^{0}_{ex}$ is in the limit circle case at $x=0$ as well.

\end{remark}

\noindent
{\bf Acknowledgements}\\[1mm]
This work has been partly supported by  the ANR-DFG projects “CoRoMo” ANR-22-CE92-0077-01. D.~Cannarsa was supported by the Academy of Finland, grant 322898 \textit{``Sub-Riemannian Geometry via Metric-geometry and Lie-group Theory''}.
E.~Pozzoli acknowledges financial support from the STARS Consolidator Grant 2021 “NewSRG” of the University of Padova, and from PNRR MUR project PE0000023-NQSTI. I.~Beschastnyi was supported through the CIDMA Center for Research and Development in Mathematics and Applications, and the Portuguese Foundation for Science and Technology (``FCT - Funda\c{c}\~ao para a Ci\^encia e a Tecnologia") within the project UIDP/04106/2020 and UIDB/04106/2020.

The authors are grateful to David Krejcirik for illuminating discussions concerning his works \cite{David2, David3}.

\bibliographystyle{amsalpha}
\bibliography{references}

\newcommand{\etalchar}[1]{$^{#1}$}
\providecommand{\bysame}{\leavevmode\hbox to3em{\hrulefill}\thinspace}
\providecommand{\MR}{\relax\ifhmode\unskip\space\fi MR }
\providecommand{\MRhref}[2]{%
  \href{http://www.ams.org/mathscinet-getitem?mr=#1}{#2}
}
\providecommand{\href}[2]{#2}
\begin{thebibliography}{DMENS80}

\bibitem[ABB20]{ABB}
A.~Agrachev, D.~Barilari, and U.~Boscain, \emph{A comprehensive introduction to
  sub-{R}iemannian geometry}, Cambridge Studies in Advanced Mathematics, vol.
  181, Cambridge University Press, Cambridge, 2020, From the Hamiltonian
  viewpoint, With an appendix by Igor Zelenko.

\bibitem[ABS08]{ABS}
A.~Agrachev, U.~Boscain, and M.~Sigalotti, \emph{{A {G}auss-{B}onnet-like
  formula on two-dimensional almost-{R}iemannian manifolds}}, Discrete Contin.
  Dyn. Syst. \textbf{20} (2008), no.~4, 801--822.

\bibitem[AD99]{driver}
L.~Andersson and B.~K. Driver, \emph{Finite-dimensional approximations to
  {W}iener measure and path integral formulas on manifolds}, J. Funct. Anal.
  \textbf{165} (1999), no.~2, 430--498.

\bibitem[Bao67]{baouendi}
Mohamed~Salah Baouendi, \emph{Sur une classe d'op\'erateurs elliptiques
  d\'eg\'en\'er\'es}, Bull. Soc. Math. France \textbf{95} (1967), 45--87.
  \MR{0228819}

\bibitem[BBP21]{Ivan-Ugo-Eugenio-2020}
I.~Beschastny\u{\i}, U.~Boscain, and E.~Pozzoli, \emph{Quantum confinement for
  the curvature {L}aplacian {$-\Delta+cK$} on 2{D}-{A}lmost-{R}iemannian
  manifolds}, Potential Anal (2021).

\bibitem[Bes23]{ivan-gruppoidi}
Ivan Beschastnyi, \emph{Closure of the {L}aplace-{B}eltrami {O}perator on 2{D}
  {A}lmost-{R}iemannian {M}anifolds and {S}emi-{F}redholm {P}roperties of
  {D}ifferential {O}perators on {L}ie {M}anifolds}, Results Math. \textbf{78}
  (2023), no.~2, Paper No. 59.

\bibitem[BL13]{Boscain-Laurent-2013}
U.~Boscain and C.~Laurent, \emph{{The {L}aplace-{B}eltrami operator in
  almost-{R}iemannian geometry}}, Ann. Inst. Fourier (Grenoble) \textbf{63}
  (2013), no.~5, 1739--1770.

\bibitem[BMM15]{beauchard1}
Karine Beauchard, Luc Miller, and Morgan Morancey, \emph{2{D} {G}rushin-type
  equations: minimal time and null controllable data}, J. Differential
  Equations \textbf{259} (2015), no.~11, 5813--5845. \MR{3397310}

\bibitem[BN{\etalchar{+}}20]{boscain-neel}
U.~Boscain, R.~W. Neel, et~al., \emph{Extensions of {B}rownian motion to a
  family of {G}rushin-type singularities}, Electronic Communications in
  Probability \textbf{25} (2020).

\bibitem[BP16]{Boscain-Prandi-JDE-2016}
U.~Boscain and D.~Prandi, \emph{{Self-adjoint extensions and stochastic
  completeness of the {L}aplace-{B}eltrami operator on conic and anticonic
  surfaces}}, J. Differential Equations \textbf{260} (2016), no.~4, 3234--3269.

\bibitem[DeW92]{driver-20}
B.~DeWitt, \emph{Supermanifolds}, second ed., Cambridge Monographs on
  Mathematical Physics, Cambridge University Press, Cambridge, 1992.

\bibitem[DMENS80]{driver-22}
C.~DeWitt-Morette, K.~D. Elworthy, B.~L. Nelson, and G.~S. Sammelman, \emph{A
  stochastic scheme for constructing solutions of the {S}chr\"{o}dinger
  equations}, Ann. Inst. H. Poincar\'{e} Sect. A (N.S.) \textbf{32} (1980),
  no.~4, 327--341.

\bibitem[FPR20]{Franceschi-Prandi-Rizzi-2017}
V.~Franceschi, D.~Prandi, and L.~Rizzi, \emph{{On the essential
  self-adjointness of singular sub-Laplacians}}, Potential Analysis \textbf{53}
  (2020), no.~1, 89--112.

\bibitem[Ful99]{fulling}
S.~A. Fulling, \emph{Pseudodifferential operators, covariant quantization, the
  inescapable {V}an {V}leck-{M}orette determinant, and the {$R/6$}
  controversy}, Relativity, particle physics and cosmology ({C}ollege
  {S}tation, {TX}, 1998), World Sci. Publ., River Edge, NJ, 1999, pp.~329--342.

\bibitem[GMP19]{GMP-Grushin-2018}
M.~Gallone, A.~Michelangeli, and E.~Pozzoli, \emph{On geometric quantum
  confinement in {G}rushin-type manifolds}, Z. Angew. Math. Phys. \textbf{70}
  (2019), no.~6, Art. 158, 17.

\bibitem[GMP22]{Gallone-Michelangeli-Pozzoli-2020}
Matteo Gallone, Alessandro Michelangeli, and E.~Pozzoli, \emph{Quantum
  geometric confinement and dynamical transmission in {G}rushin cylinder},
  Reviews in Mathematical Physics \textbf{34} (2022), no.~07, 2250018.

\bibitem[Gra04]{tubes}
Alfred Gray, \emph{Tubes}, second ed., Progress in Mathematics, vol. 221,
  Birkh\"{a}user Verlag, Basel, 2004, With a preface by Vicente Miquel.
  \MR{2024928}

\bibitem[Gru70]{grusin1}
V.~V. Gru{\v{s}}in, \emph{A certain class of hypoelliptic operators}, Mat. Sb.
  (N.S.) \textbf{83 (125)} (1970), 456--473. \MR{MR0279436 (43 \#5158)}

\bibitem[Kre14]{David2}
D.~Krejcirik, \emph{Spectrum of the {L}aplacian in narrow tubular
  neighbourhoods of hypersurfaces with combined {D}irichlet and {N}eumann
  boundary conditions}, Mathematica Bohemica \textbf{139} (2014), 185--193.

\bibitem[KRT14]{David3}
D.~Krej{\v c}i{\v r}{\'\i}k, N.~Raymond, and M.~Tu{\v s}ek, \emph{The magnetic
  {L}aplacian in shrinking tubular neighborhoods of hypersurfaces}, The Journal
  of Geometric Analysis \textbf{25} (2014).

\bibitem[LS23]{cirillo}
Cyril Letrouit and Chenmin Sun, \emph{Observability of
  {B}aouendi-{G}rushin-type equations through resolvent estimates}, J. Inst.
  Math. Jussieu \textbf{22} (2023), no.~2, 541----579. \MR{4557902}

\bibitem[LTW11]{Lampart}
J.~Lampart, S.~Teufel, and J.~Wachsmuth, \emph{Effective {H}amiltonians for
  thin {D}irichlet tubes with varying cross-section}, Mathematical results in
  quantum physics, World Sci. Publ., Hackensack, NJ, 2011, pp.~183--189.

\bibitem[PRS18]{Prandi-Rizzi-Seri-2016}
D.~Prandi, L.~Rizzi, and M.~Seri, \emph{Quantum confinement on non-complete
  {R}iemannian manifolds}, J. Spectr. Theory \textbf{8} (2018), no.~4,
  1221--1280.

\bibitem[RS75]{rs2}
M.~Reed and B.~Simon, \emph{{Methods of modern mathematical physics. {II}.
  {F}ourier analysis, self-adjointness}}, Academic Press [Harcourt Brace
  Jovanovich, Publishers], New York-London, 1975. \MR{0493420 (58 \#12429b)}

\bibitem[Woo92]{driver-97}
N.~M.~J. Woodhouse, \emph{Geometric quantization}, second ed., Oxford
  Mathematical Monographs, The Clarendon Press, Oxford University Press, New
  York, 1992, Oxford Science Publications. \MR{1183739}

\end{thebibliography}

\begin{figure}[ht]
\includegraphics[width=0.5\linewidth]{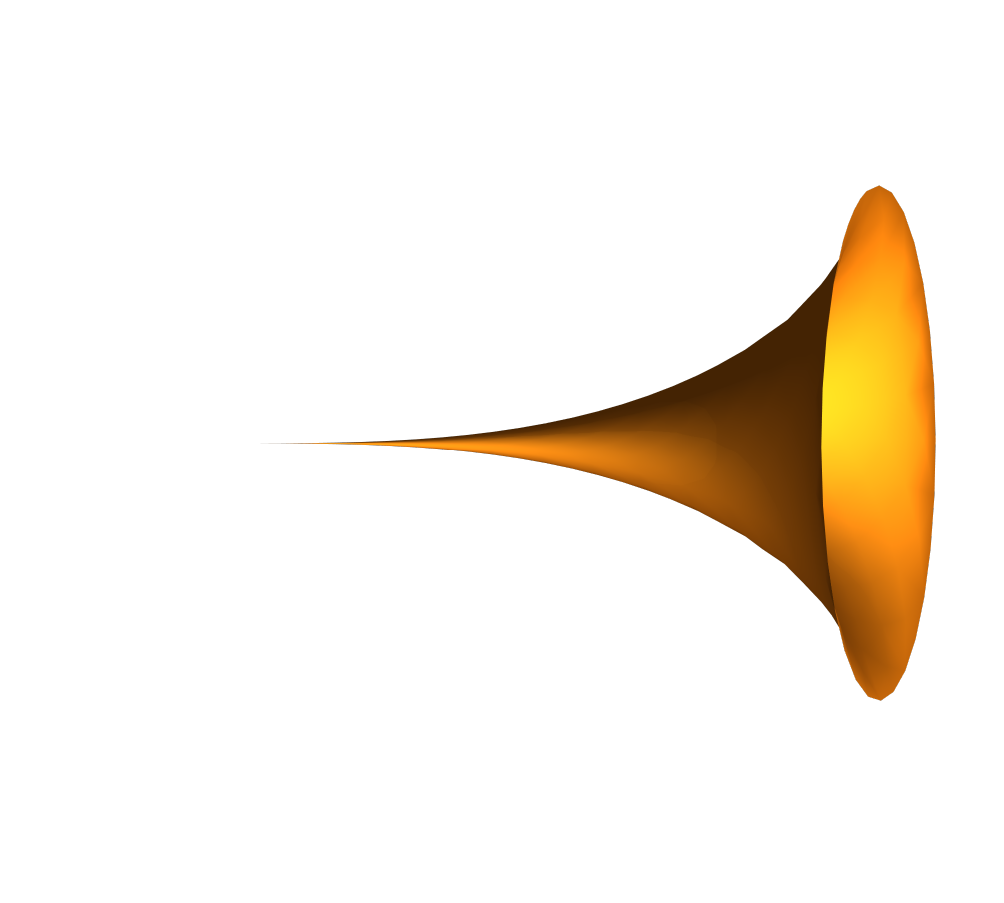}~\includegraphics[width=0.5\linewidth]{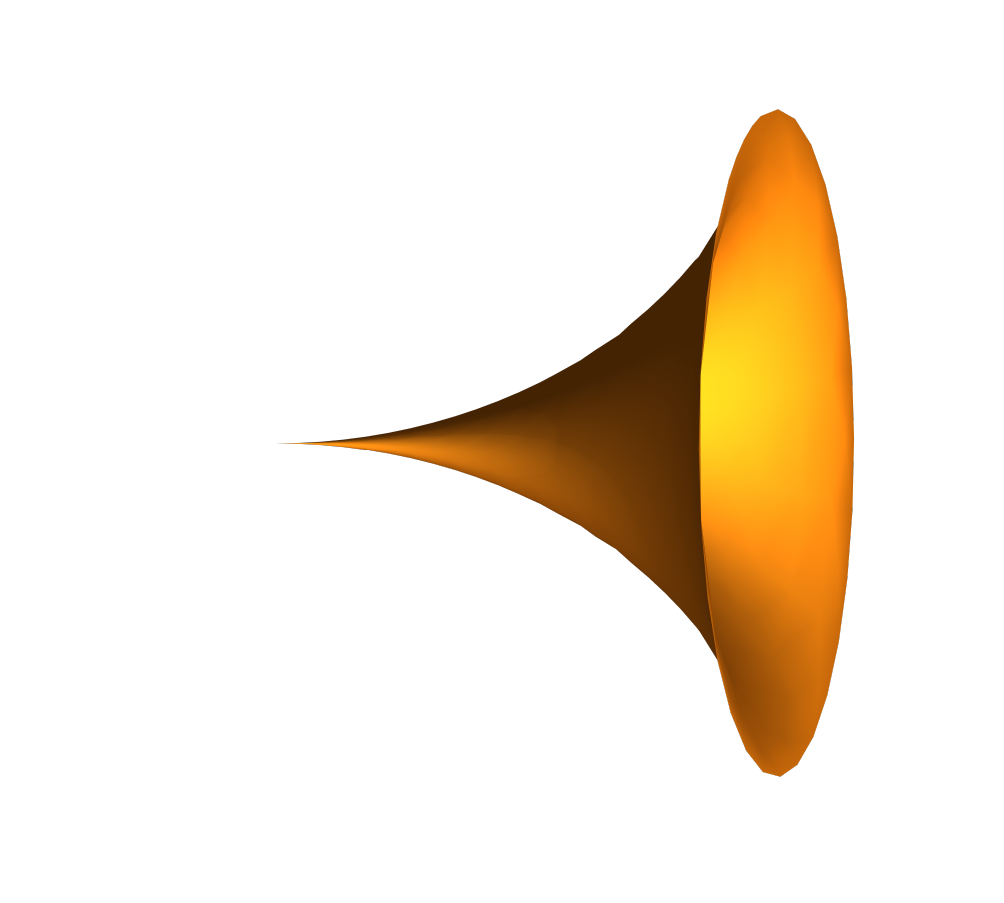} 

~~~~~~~~~~~~~~~~~~~~~~~~~~~~~~~~~~~$\alpha=-3$~~~~~~~~~~~~~~~~~~~~~~~~~~~~~~~~~~~~~~~~~~~~~~~~~~~~~~~~~~~~~~~~~~$\alpha=-2$

\caption{Embeddings of the $\alpha$-Grushin cylinder, for $\alpha<-1$. The cusp corresponds  the point $x=0$.  \label{falfa-1}}
\end{figure}

\begin{figure}[ht]

\includegraphics[width=0.33\linewidth]{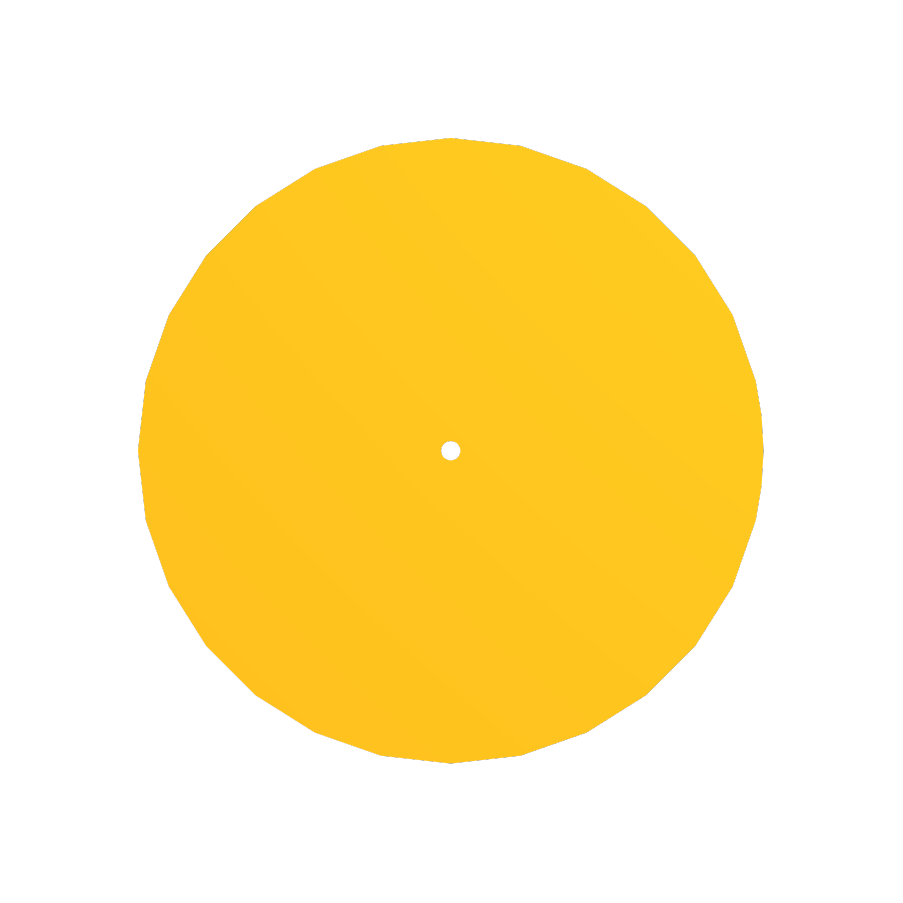}~~~~~~~~~~~~~~~~~~~~~~~~~~~~~~~~~~~\includegraphics[width=0.33\linewidth]{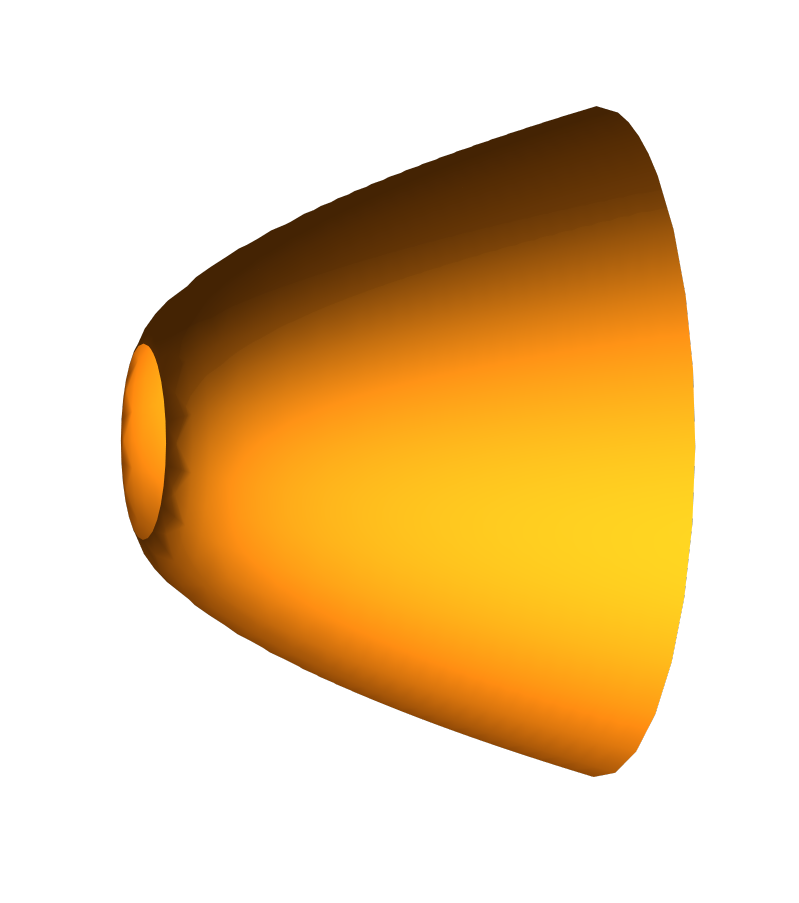}

~~~~~~~~~~~~~~~~~~~~~~~~~~~~~~~~~~~$\alpha=-1$~~~~~~~~~~~~~~~~~~~~~~~~~~~~~~~~~~~~~~~~~~~~~~~~~~~~~~~~~~~~~~~$\alpha=-1/2$

\caption{Embeddings of the $\alpha$-Grushin cylinder, for $\alpha\in[-1,0)$.   \label{falfa-2}}

\end{figure}

\begin{figure}[ht]

\includegraphics[width=0.33\linewidth]{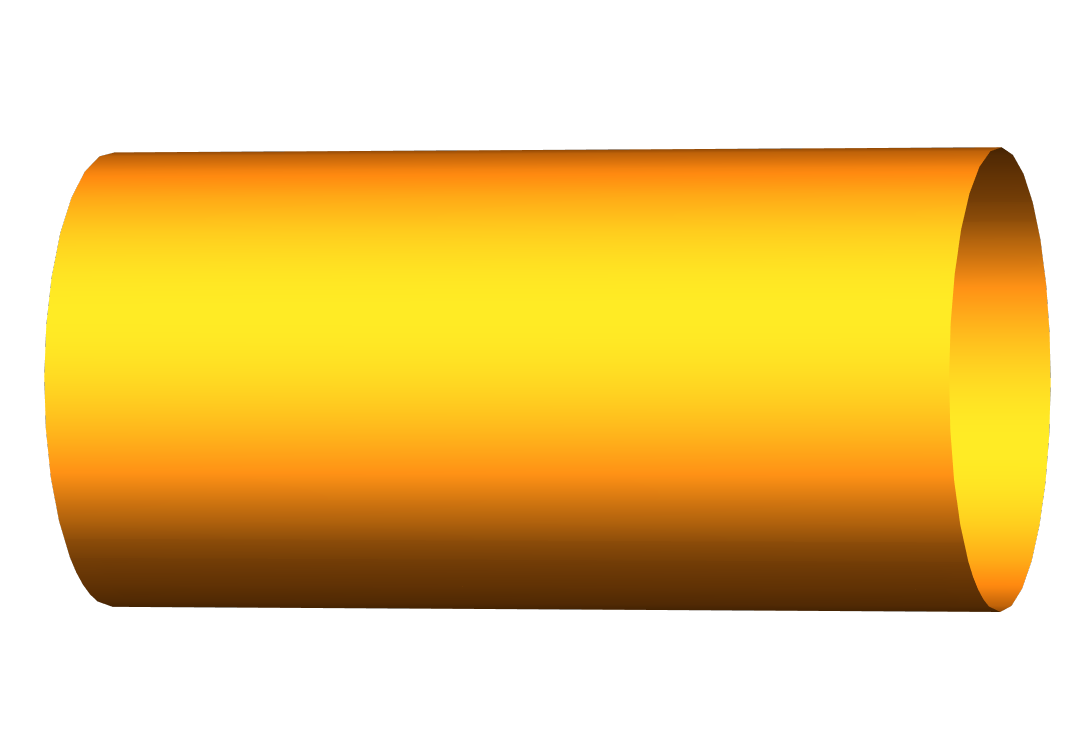}~~~~~~~~~~~~~~~~~~~~~~~~~~~~~~~~~~~\includegraphics[width=0.33\linewidth]{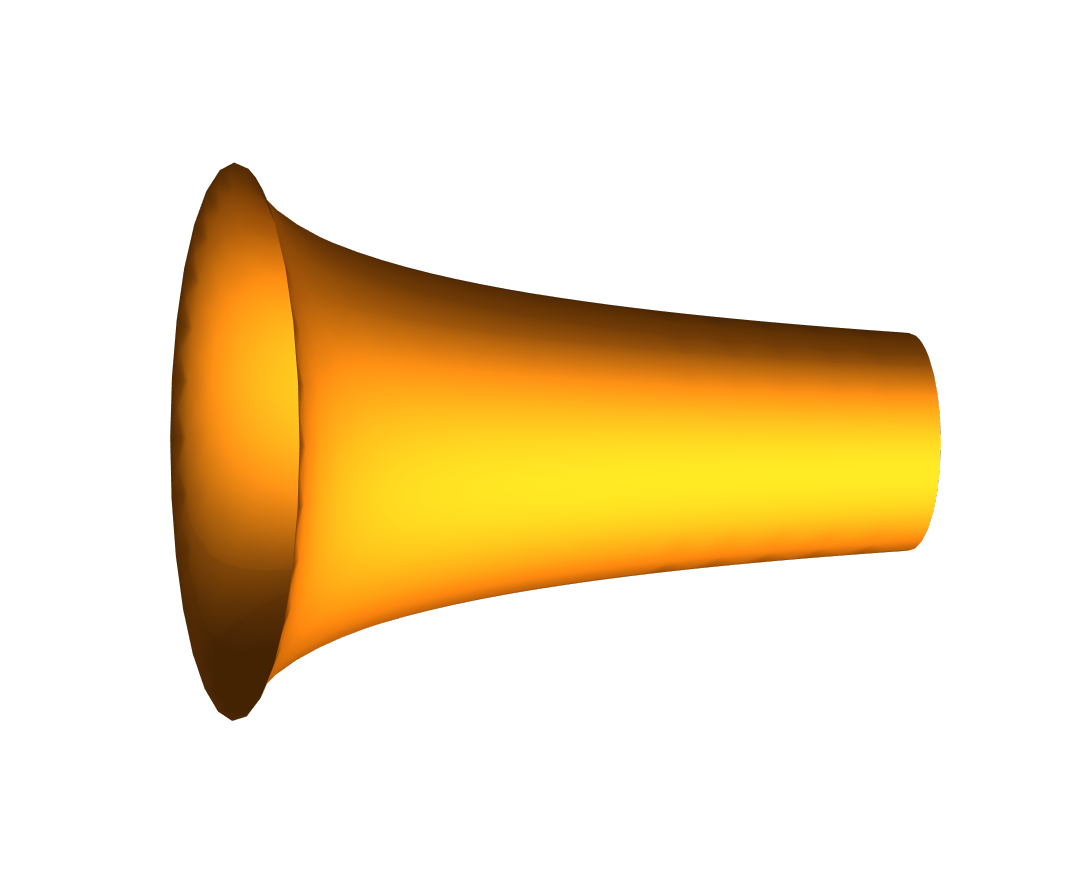}

~~~~~~~~~~~~~~~~~~~~~~~~~~~~~~~~~~~$\alpha=0$~~~~~~~~~~~~~~~~~~~~~~~~~~~~~~~~~~~~~~~~~~~~~~~~~~~~~~~~~~~~~~~$\alpha=1/2$
\caption{Embeddings of the $\alpha$-Grushin cylinder, for $\alpha\in[0,1)$.   \label{falfa-3}}
\end{figure}

\begin{figure}[ht]

\includegraphics[width=0.5\linewidth]{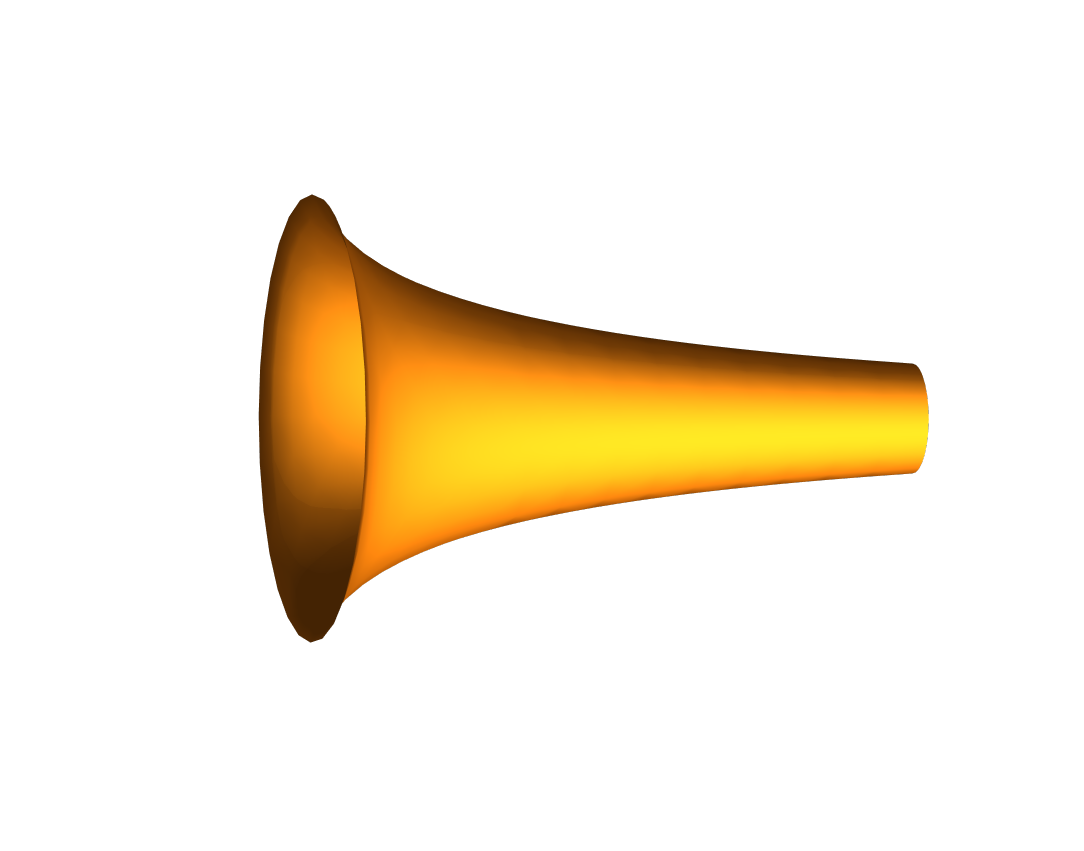}~~~~\includegraphics[width=0.5\linewidth]{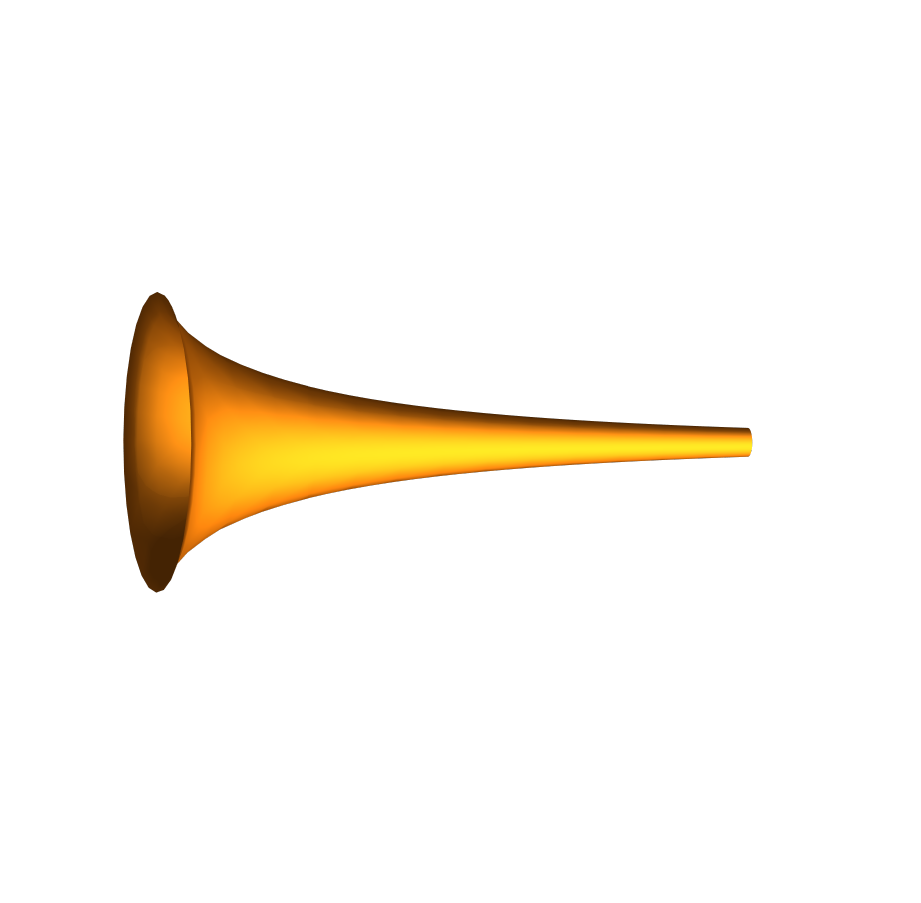} 

~~~~~~~~~~~~~~~~~~~~~~~~~~~~~~~~~~~$\alpha=1$~~~~~~~~~~~~~~~~~~~~~~~~~~~~~~~~~~~~~~~~~~~~~~~~~~~~~~~~~~~~~$\alpha=2$

\caption{Embeddings of the $\alpha$-Grushin cylinder, for $\alpha\geq1$.   \label{falfa-4}}
\end{figure}

 \end{document}